\documentclass[11pt,reqno]{amsart}

\usepackage[T1]{fontenc}
\usepackage{lmodern}
\usepackage{microtype}
\usepackage{amsmath,amssymb,amsthm,mathtools}
\usepackage{enumitem}
\usepackage{xcolor}
\usepackage[colorlinks=true,linkcolor=blue!55!black,citecolor=blue!55!black,urlcolor=blue!55!black]{hyperref}
\hypersetup{
  pdftitle={Rectangular Pegs on Jordan Curves of Finite p-Variation},
  pdfauthor={Xiangfei Li},
  pdfsubject={Square and rectangular peg problems for rough Jordan curves},
  pdfkeywords={square peg problem, rectangular peg problem, Jordan curve, p-variation, Young integral}
}

\newtheorem{theorem}{Theorem}[section]
\newtheorem{proposition}[theorem]{Proposition}
\newtheorem{lemma}[theorem]{Lemma}
\newtheorem{example}[theorem]{Example}
\newtheorem{corollary}[theorem]{Corollary}
\theoremstyle{definition}

\theoremstyle{remark}

\newcommand{\R}{\mathbb{R}}

\newcommand{\T}{\mathbb{T}}

\title[Rectangular pegs on finite-variation Jordan curves]
{Rectangular Pegs on Jordan Curves of Finite \(p\)-Variation}

\author{Xiangfei Li}
\address{Soochow Academy\\
 Suzhou 215006\\ China}
\email{xfli@amss.ac.cn}

\author{Yichen Pan}
\address{Qiuzhen College\\Tsinghua Univerisity\\ Beijing 100084\\ China}
\email{panyc26@mails.tsinghua.edu.cn}


\begin{document}

\begin{abstract}
We prove that every planar Jordan curve of finite \(p\)-variation, with
\(1\leq p<2\), inscribes a rectangle of every prescribed similarity class.
In particular, every such curve inscribes a square.  The proof combines the
recent criterion of Asano and Ike with a
variation-controlled approximation argument.  We show that for every
\(q>p\), a Jordan curve of finite \(p\)-variation can be approximated in the
\(q\)-variation topology by smooth Jordan embeddings.  The construction uses
simple polygonal interpolants of Boedihardjo and Geng, an elementary
interpolation inequality between variation seminorms, and a
variation-controlled smoothing of polygonal embeddings.  Young integration
then gives locally uniform convergence of the associated primitives.

\end{abstract}

\maketitle

\section{Introduction}

The square peg problem, attributed to Toeplitz, asks whether every planar
Jordan curve contains the four vertices of a nondegenerate square.  The
problem remains open for arbitrary Jordan curves; see the survey of Matschke
\cite{Matschke2014}.  Positive results have traditionally imposed geometric
or analytic regularity.  Stromquist proved the conjecture for locally
monotone curves \cite{Stromquist1989}, Tao established it for certain unions
of Lipschitz graphs \cite{Tao2017}, and Greene and Lobb proved that every
smooth Jordan curve inscribes every rectangle \cite{GreeneLobb2021}.
For \(\theta\in(0,\pi)\), a \(\theta\)-rectangle means a rectangle whose
diagonals meet at angle \(\theta\); prescribing \(\theta\), up to replacing it
by \(\pi-\theta\), is equivalent to prescribing the similarity class.

Recently, Asano and Ike introduced a  criterion
which applies to continuous Jordan curves admitting what they call a
\emph{continuous Legendrian lift} \cite{AsanoIke}.  Their theorem implies the
rectangular peg conclusion for all rectifiable Jordan curves and all locally
monotone Jordan curves.  The purpose of this note is to verify their criterion
for a natural class of rough curves arising in deterministic rough-path
theory.

For a continuous path \(x\colon[0,1]\to\R^d\) and \(r\geq1\), write
\[
 \lVert x\rVert_{r\text{-var};[0,1]}
 :=\left(
 \sup_{\mathcal D}\sum_{j}|x(t_j)-x(t_{j-1})|^r
 \right)^{1/r},
\]
where the supremum ranges over all finite partitions
\(\mathcal D=\{0=t_0<\cdots<t_m=1\}\).  Our main result is the following.

\begin{theorem}\label{thm:main}
Let \(C\subset\R^2\) be a Jordan curve admitting a parametrization of finite
\(p\)-variation for some \(1\leq p<2\).  Then \(C\) inscribes a rectangle of
every prescribed similarity class.  In particular, \(C\) inscribes a square.
\end{theorem}

The key point is not merely uniform approximation.  Smooth approximating
curves have inscribed rectangles, but those rectangles may collapse to a
point in the limit.  The Asano--Ike criterion requires convergence of
primitives of the Liouville form.  We obtain this convergence by constructing
smooth Jordan approximants in \(q\)-variation for some \(q<2\), after which
Young's integration theorem applies \cite{Young1936}.

We identify \(\T\) with \([0,1]/\{0,1\}\). The approximation statement may also be useful independently.

\begin{theorem}\label{thm:approx-intro}
Let \(\gamma\colon\T\to\R^2\) be a Jordan parametrization of finite
\(p\)-variation.  For every \(q>p\), there are smooth Jordan embeddings
\(\gamma_n\colon\T\to\R^2\) such that
\[
 |\gamma_n(0)-\gamma(0)|
 +\lVert\gamma_n-\gamma\rVert_{q\text{-var}}
 \longrightarrow0.
\]
\end{theorem}

The strict inequality \(q>p\) is essential to the proof: uniform convergence
combined with a uniform \(p\)-variation bound improves to convergence in every
larger variation exponent, but in general not in the original exponent.

\section{Variation estimates}

We begin with two elementary facts.  All variation seminorms in this section
are taken on \([0,1]\), unless another interval is displayed.

\begin{lemma}\label{lem:interpolation}
Let \(1\leq p<q<\infty\), and let \(h\colon[0,1]\to\R^d\) be continuous with
finite \(p\)-variation.  Then
\begin{equation}\label{eq:variation-interpolation}
 \lVert h\rVert_{q\text{-var}}
 \leq
 \bigl(2\lVert h\rVert_\infty\bigr)^{1-p/q}
 \lVert h\rVert_{p\text{-var}}^{p/q}.
\end{equation}
\end{lemma}

\begin{proof}
For every partition \(\mathcal D=\{t_j\}\),
\[
\begin{aligned}
 \sum_j |h(t_j)-h(t_{j-1})|^q
 &\leq
 \left(\max_j|h(t_j)-h(t_{j-1})|\right)^{q-p}
 \sum_j|h(t_j)-h(t_{j-1})|^p\\
 &\leq
 (2\lVert h\rVert_\infty)^{q-p}
 \lVert h\rVert_{p\text{-var}}^p.
\end{aligned}
\]
Taking the supremum and then the \(q\)-th root proves
\eqref{eq:variation-interpolation}.
\end{proof}

Let \(x\colon[0,1]\to\R^d\) be continuous and let
\(\mathcal P=\{0=t_0<\cdots<t_N=1\}\).  Denote by \(x^{\mathcal P}\) its
piecewise affine interpolation through the values \(x(t_i)\).

\begin{lemma}\label{lem:polygon-var}
For every \(p\geq1\),
\begin{equation}\label{eq:polygon-var}
 \lVert x^{\mathcal P}\rVert_{p\text{-var}}
 \leq \lVert x\rVert_{p\text{-var}}.
\end{equation}
Moreover, if \(\omega_x\) is the modulus of continuity of \(x\), then
\begin{equation}\label{eq:polygon-uniform}
 \lVert x^{\mathcal P}-x\rVert_\infty
 \leq \omega_x(|\mathcal P|).
\end{equation}
\end{lemma}

\begin{proof}
On each interval \([t_{i-1},t_i]\), the image of
\(x^{\mathcal P}\) is a line segment.  Given a partition used to compute its
\(p\)-variation, consider a partition point lying in the interior of one such
segment.  With its two neighboring sampled values fixed, its contribution is
of the form
\[
 z\longmapsto |z-a|^p+|b-z|^p.
\]
This is convex along the segment, so moving the point to one of the two
endpoints cannot decrease the sum.  Iterating, the variation sum is bounded
by one using only vertices of the interpolation.  Those vertices are values
of \(x\) at the times \(t_i\), which proves \eqref{eq:polygon-var}.

If \(t\in[t_{i-1},t_i]\) and
\(x^{\mathcal P}(t)=(1-\lambda)x(t_{i-1})+\lambda x(t_i)\), then
\[
 |x^{\mathcal P}(t)-x(t)|
 \leq
 (1-\lambda)|x(t_{i-1})-x(t)|
 +\lambda|x(t_i)-x(t)|
 \leq\omega_x(|\mathcal P|),
\]
which proves \eqref{eq:polygon-uniform}.
\end{proof}

\section{Smooth approximation through Jordan embeddings}

 The following result of
Boedihardjo and Geng is the topological input: for every Jordan
parametrization \(\gamma\colon\T\to\R^2\) and every \(\varepsilon>0\), there
is a partition of mesh less than \(\varepsilon\) whose piecewise affine
interpolation is again a Jordan curve \cite[Theorem~2.2]{BoedihardjoGeng2015}.
The fact that this is an interpolation, rather than only a uniformly close
polygonal curve, is crucial for controlling variation.

We shall also need a quantitative smoothing observation.

\begin{lemma}\label{lem:smooth-polygon}
Let \(P\colon\T\to\R^2\) be a parametrized Jordan polygon which is affine and
nonconstant on each edge.  For every \(\eta>0\), there exists a smooth Jordan
embedding \(S\colon\T\to\R^2\) such that \(S(0)=P(0)\) and
\begin{equation}\label{eq:w11-smooth}
 \lVert S-P\rVert_{1\text{-var}}<\eta.
\end{equation}
\end{lemma}

\begin{proof}
Extend \(P\) periodically to \(\R\), let \(\rho_\delta\) be a nonnegative
periodic smooth approximate identity, and put
\[
 P_\delta=P*\rho_\delta.
\]
Since \(P\in W^{1,1}(\T;\R^2)\),
\[
 P_\delta' =P'*\rho_\delta\longrightarrow P'
 \quad\text{in }L^1(\T;\R^2).
\]
Consequently,
\begin{equation}\label{eq:w11-convolution}
 \lVert P_\delta-P\rVert_{1\text{-var}}
 =\int_{\T}|P_\delta'(t)-P'(t)|\,dt
 \longrightarrow0.
\end{equation}
Also, \(P_\delta\to P\) uniformly.

We claim that \(P_\delta\) is a smooth embedding for all sufficiently small
\(\delta\).  Away from the vertices, its derivative is the constant nonzero
velocity of an edge.  At a vertex, let \(v_-\) and \(v_+\) be the two adjacent
edge velocities.  They cannot be negative scalar multiples of each other,
since that would force the adjacent edges to overlap and contradict the
injectivity of \(P\).  Hence there is a linear functional \(\ell\) with
\[
 \ell(v_-)>0,
 \qquad
 \ell(v_+)>0.
\]
If \(\delta\) is smaller than half the least parameter length of an edge, then
near this vertex \(P_\delta'\) is a nonnegative weighted average of \(v_-\) and
\(v_+\).  Thus \(\ell(P_\delta')>0\), so \(P_\delta\) is regular and locally
injective there.  The same argument on the interiors of the finitely many
edges gives a finite collection of parameter intervals on each of which a
linear projection of \(P_\delta\) is strictly monotone, uniformly for small
\(\delta\).  Therefore there exists \(r>0\) such that
\begin{equation}\label{eq:local-inj}
 P_\delta(s)\neq P_\delta(t)
 \quad\text{whenever }0<d_{\T}(s,t)<r
\end{equation}
for all sufficiently small \(\delta\).

For pairs separated in parameter, compactness and injectivity of \(P\) give
\[
 m_r:=\min_{d_{\T}(s,t)\geq r}|P(s)-P(t)|>0.
\]
Uniform convergence implies
\(\lVert P_\delta-P\rVert_\infty<m_r/3\) for small \(\delta\), and hence
\[
 |P_\delta(s)-P_\delta(t)|
 \geq m_r-2\lVert P_\delta-P\rVert_\infty>0
\]
whenever \(d_{\T}(s,t)\geq r\).  Together with
\eqref{eq:local-inj}, this proves global injectivity.

Finally define
\[
 S(t)=P_\delta(t)-P_\delta(0)+P(0).
\]
Translation preserves regularity and injectivity, while \(S(0)=P(0)\) and
\(\lVert S-P\rVert_{1\text{-var}}\) equals the left-hand side of
\eqref{eq:w11-convolution}.  Choosing \(\delta\) sufficiently small proves
the lemma.
\end{proof}

\begin{proposition}
\label{prop:smooth-approx}
Let \(\gamma\colon\T\to\R^2\) be a Jordan parametrization of finite
\(p\)-variation.  For every \(q>p\), there exist smooth Jordan embeddings
\(\gamma_n\colon\T\to\R^2\), with \(\gamma_n(0)=\gamma(0)\), such that
\begin{equation}\label{eq:qvar-convergence}
 \lVert\gamma_n-\gamma\rVert_{q\text{-var}}\longrightarrow0.
\end{equation}
\end{proposition}

\begin{proof}
By \cite[Theorem~2.2]{BoedihardjoGeng2015}, choose partitions
\(\mathcal P_n\) with \(|\mathcal P_n|\to0\) such that
\(P_n:=\gamma^{\mathcal P_n}\) is a Jordan polygon.  Lemma
\ref{lem:polygon-var} gives
\[
 \lVert P_n\rVert_{p\text{-var}}
 \leq\lVert\gamma\rVert_{p\text{-var}},
 \qquad
 \lVert P_n-\gamma\rVert_\infty
 \leq\omega_\gamma(|\mathcal P_n|)\longrightarrow0.
\]
The triangle inequality for \(p\)-variation yields
\[
 \lVert P_n-\gamma\rVert_{p\text{-var}}
 \leq2\lVert\gamma\rVert_{p\text{-var}}.
\]
By Lemma \ref{lem:interpolation},
\begin{equation}\label{eq:polygon-qvar}
 \lVert P_n-\gamma\rVert_{q\text{-var}}
 \leq
 \bigl(2\omega_\gamma(|\mathcal P_n|)\bigr)^{1-p/q}
 \bigl(2\lVert\gamma\rVert_{p\text{-var}}\bigr)^{p/q}
 \longrightarrow0.
\end{equation}

Apply Lemma \ref{lem:smooth-polygon} to obtain a smooth Jordan embedding
\(\gamma_n\), based at \(\gamma(0)\), with
\[
 \lVert\gamma_n-P_n\rVert_{1\text{-var}}<\frac1n.
\]
Since the \(q\)-variation seminorm is bounded above by the total variation,
\[
 \lVert\gamma_n-\gamma\rVert_{q\text{-var}}
 \leq\frac1n+\lVert P_n-\gamma\rVert_{q\text{-var}}
 \longrightarrow0.
\]
This proves the proposition and Theorem \ref{thm:approx-intro}.
\end{proof}

\section{Young primitives and the rectangular peg theorem}

We recall the Young--Loeve estimate in the equal-exponent form needed below.
If \(1\leq q<2\) and \(u,v\) have finite \(q\)-variation, then the Young
integral \(\int u\,dv\) exists and
\begin{equation}\label{eq:young-loeve}
 \left|
 \int_s^t u\,dv-u(s)(v(t)-v(s))
 \right|
 \leq C_q
 \lVert u\rVert_{q\text{-var};[s,t]}
 \lVert v\rVert_{q\text{-var};[s,t]}.
\end{equation}
See Young's original paper \cite{Young1936}.

\begin{lemma}\label{lem:young-stability}
Let \(1\leq q<2\), and suppose
\(c_n=(x_n,y_n)\) and \(c=(x,y)\) are continuous planar paths such that
\[
 |c_n(0)-c(0)|+
 \lVert c_n-c\rVert_{q\text{-var}}\longrightarrow0.
\]
Define
\[
 f_n(t)=\int_0^t y_n\,dx_n,
 \qquad
 f(t)=\int_0^t y\,dx.
\]
Then \(f_n\to f\) uniformly on \([0,1]\).
\end{lemma}

\begin{proof}
The assumed convergence implies uniform convergence and uniform boundedness
of the \(q\)-variation norms.  By bilinearity of the Young integral,
\[
 f_n(t)-f(t)
 =\int_0^t(y_n-y)\,dx_n
  +\int_0^t y\,d(x_n-x).
\]
Applying \eqref{eq:young-loeve} to the two terms gives
\[
\begin{aligned}
 |f_n(t)-f(t)|
 \leq{}&
 \lVert y_n-y\rVert_\infty
 \lVert x_n\rVert_{q\text{-var}}
 +C_q\lVert y_n-y\rVert_{q\text{-var}}
       \lVert x_n\rVert_{q\text{-var}}\\
 &+\lVert y\rVert_\infty
 \lVert x_n-x\rVert_{q\text{-var}}
 +C_q\lVert y\rVert_{q\text{-var}}
       \lVert x_n-x\rVert_{q\text{-var}}.
\end{aligned}
\]
The right-hand side is independent of \(t\) and tends to zero.
\end{proof}

We use the following consequence of the main theorem of Asano and Ike
\cite{AsanoIke}.

\begin{theorem}[{\cite[Theorem~1.1]{AsanoIke}}]\label{thm:AI}
Let \(c\colon\T\to\R^2\) be a Jordan parametrization.  Suppose there are
smooth Jordan parametrizations \(c_n\to c\) uniformly such that, after lifting
the parameter periodically to \(\R\), primitives of \(c_n^*\lambda\) converge
locally uniformly on \(\R\), where \(\lambda=y\,dx\) is the standard
Liouville form.  Then \(c(\T)\) inscribes a rectangle of every prescribed
similarity class.
\end{theorem}

\begin{proof}[Proof of Theorem \ref{thm:main}]
Choose \(q\) with
\[
 p<q<2.
\]
Proposition \ref{prop:smooth-approx} provides smooth Jordan embeddings
\(\gamma_n=(x_n,y_n)\) converging to
\(\gamma=(x,y)\) in \(q\)-variation.  Lemma \ref{lem:young-stability} shows
that the primitives
\[
 f_n(t)=\int_0^t y_n\,dx_n
\]
converge uniformly on one period to the Young primitive
\(f(t)=\int_0^t y\,dx\).  Periodicity of the curves and convergence of the
one-period integrals imply locally uniform convergence of the lifted
primitives on \(\R\): explicitly, if \(t=k+s\), with \(k\in\mathbb Z\) and
\(s\in[0,1)\), then the normalized lift is
\(\widetilde f_n(t)=k f_n(1)+f_n(s)\).  Thus the hypotheses of Theorem
\ref{thm:AI} hold, and the desired rectangular inscriptions follow.
\end{proof}

\begin{corollary}\label{cor:holder}
Let \(\gamma\colon\T\to\R^2\) be an \(\alpha\)-H\"older Jordan
parametrization with \(1/2<\alpha\leq1\).  Then \(\gamma(\T)\) inscribes a
rectangle of every prescribed similarity class.
\end{corollary}

\begin{proof}
Set \(p=1/\alpha<2\).  For every partition \(\{t_j\}\),
\[
 \sum_j|\gamma(t_j)-\gamma(t_{j-1})|^p
 \leq [\gamma]_{C^\alpha}^p
 \sum_j|t_j-t_{j-1}|^{\alpha p}
 = [\gamma]_{C^\alpha}^p.
\]
Thus \(\gamma\) has finite \(p\)-variation, and Theorem \ref{thm:main}
applies.
\end{proof}

\section{Strictness of the finite-variation class}

Finite \(1\)-variation is equivalent to rectifiability.  For
\(1<p<2\), Theorem \ref{thm:main} also covers curves of infinite length.  The
following example shows that the finite-\(p\)-variation class is not contained
in the locally monotone class of Stromquist.

\begin{example}\label{prop:spiral}
For every \(1<p<2\), there exists a Jordan curve which
\begin{enumerate}[label=\textup{(\roman*)}]
 \item has finite \(p\)-variation;
 \item has infinite length; and
 \item is not locally monotone.
\end{enumerate} 
\end{example}

\begin{proof}
Choose
\[
 \frac1p<\alpha<1
\]
and a fixed \(\vartheta\in(0,\pi)\).  Consider the two spirals
\[
 z_\pm(u)=u^{-\alpha}e^{i(u\pm\vartheta/2)},
 \qquad 1\leq u<\infty,
\]
with \(z_\pm(\infty)=0\).  They are disjoint away from the origin: equality
of two points first forces equality of their moduli and hence equality of the
parameters, after which the phases differ by \(\vartheta\).  Join
\(z_+(1)\) to \(z_-(1)\) by the corresponding arc of the unit circle.  The
two spirals, the origin, and this outer arc form a Jordan curve \(C_{\alpha,
\vartheta}\).

To estimate variation, divide each spiral into half-open turns
\(I_k=[2\pi k,2\pi(k+1))\), omitting finitely many initial terms.  Its length
on the closure of \(I_k\) is at most \(C_\alpha k^{-\alpha}\).  Fix an
arbitrary finite partition.  The increments whose endpoints both lie in
\(I_k\) contribute at most \(C_\alpha^p k^{-\alpha p}\), since their total
length is at most the length of that turn.  Among the remaining increments,
at most one starts in each \(I_k\); such an increment has size at most
\(2(2\pi k)^{-\alpha}\), even if it crosses several turn boundaries.
Therefore, uniformly over all partitions,
\[
 \lVert z_\pm\rVert_{p\text{-var}}^p
 \leq C\sum_{k=1}^\infty k^{-\alpha p}<\infty.
\]
The outer circular arc has bounded variation, proving (i).

On the other hand,
\[
 |z_\pm'(u)|
 =u^{-\alpha}\sqrt{1+\alpha^2u^{-2}},
\]
so the length is infinite because \(\alpha<1\).  Finally, for every unit
vector of argument \(\phi\), the corresponding linear projection of either
spiral is
\[
 u^{-\alpha}\cos(u\pm\vartheta/2-\phi),
\]
which oscillates infinitely often on every tail. Hence, no linear projection
is strictly monotone in a neighborhood of the origin, and the curve is not
locally monotone in the sense of \cite{Stromquist1989}.
\end{proof}

\section{Acknowledgments}
The authors used ChatGPT 5.6 to polish the writing of the article. 




\end{document}